\documentclass[12pt]{article}

\usepackage{amsmath, amsthm, amsfonts, amssymb, mathtools, bbm, bm, graphicx, tikz, mathrsfs}
\usepackage[font=small]{caption}
\usepackage{newtxtext}
\usepackage[round]{natbib}
\usepackage{subcaption}
\usepackage{float}
\usepackage{setspace}
\usetikzlibrary{matrix}
\usepackage[export]{adjustbox}
\usepackage{lscape}
\usepackage{booktabs}
\usepackage{authblk}
\usepackage{algorithmicx, algpseudocode}
\usepackage{hyperref}
\usepackage{enumerate}
\usepackage[rmargin=3.3cm, lmargin=3.3cm]{geometry}
\usepackage{comment}

\hypersetup{
    colorlinks,
    citecolor=black,
    filecolor=black,
    linkcolor=black,
    urlcolor=black
}
\newcommand{\F}{\mathcal{F}}

\newcommand{\reals}{\mathbb{R}}

\newcommand{\borel}{\mathcal{B}}

\newcommand{\laweq}{\overset{Law}{=}}

\newcommand{\E}{\mathbb{E}}
\newcommand{\p}{\mathbb{P}}

\newcommand{\Q}{\mathbb{Q}}
\newcommand{\mat}{}
\newcommand{\tr}{Tr}

\newcommand{\filteredspace}{\left( \Omega, \F, \left( \F_t \right)_{ t \geq 0 }, \p \right)}

\newcommand{\id}{\mathbf{id}}

\newtheorem{theorem}{Theorem}[section]

\newtheorem{lemma}{Lemma}[section]
\newtheorem{corollary}{Corollary}[section]

\theoremstyle{definition}
\newtheorem{mydef}{Definition}[section]

\numberwithin{equation}{section}
\theoremstyle{plain}

\allowdisplaybreaks

\title{Integrated Wishart bridge processes and generalised Hartman-Watson law}
\author{Jason Leung \footnote{School of Mathematics and Statistics, e-mail: jason.leung@unimelb.edu.au} \footnote{Supported by the Albert Shimmins fund} }
\affil{\it The University of Melbourne}
\date{June, 2019}

\begin{document}
\maketitle
\begin{abstract}
This article is concerned with the joint law of an integrated Wishart bridge process and the trace of an integrated inverse Wishart bridge process over the interval $ \left[0,t\right] $. Its Laplace transform is obtained by studying the Wishart bridge processes and the absolute continuity property of Wishart laws. 
\end{abstract}

\newpage

\pagenumbering{arabic}

\section{Introduction}

Suppose $ X $ is a solution to the stochastic differential equation on the cone $ \tilde{\mathcal{S}}_n^+ $ of $ n \times n $ symmetric positive semi-definite matrices
\begin{align} \label{Wishart SDE}
	d X_t = \sqrt{X}_t d W_t a + a^\top d W^\top_t \sqrt{X}_t + \left( b X_t + X_t b + \alpha a^\top a \right)dt, \quad t \geq 0,
\end{align}
where $ X_0 = x \in \tilde{\mathcal{S}}^+_n $, $ W $ is an $ n \times n $ matrix-valued Brownian motion, $ a $ in the space $ GL(n) $ of $ n \times n $ invertible matrices, $ b $ in the cone $ \tilde{\mathcal{S}}^- $ of negative semi-definite matrices such that $ ab = ba $ and $ \alpha \in \{1, 2, \dots, n-1\} \cup (n-1, \infty) $. 

The process $ X $ satisfying (\ref{Wishart SDE}), first introduced in \cite{bru1991wishart}, is called a {\it Wishart process} of {\it dimension} $ n $, {\it index} $ \alpha $ and parameters $ a,b $ with initial value $ x $ and is denoted $ WIS(n,\alpha,a,b,x) $. It was shown in \cite{Cuchiero2011} that the stochastic differential equation (\ref{Wishart SDE}) has a unique weak solution for $ \alpha \geq n-1 $ as well as for $ \alpha \in \{1,2,\dots, n-1\} $ with the additional condition of $ \text{rank}(x) \leq \alpha $. For $ \alpha \geq n + 1 $, \cite{MAYERHOFER20112072} showed that the solution to (\ref{Wishart SDE}) exists as a strong solution and is unique for $ t \geq 0 $. Moroever, it was shown in \cite{MAYERHOFER20112072} that if the initial value $ x $ belongs to the space $ \mathcal{S}_n^+ $ of $ n \times n $ positive definite matrices, the solution to (\ref{Wishart SDE}) also belongs to $ \mathcal{S}_n^+ $.

Given $ \alpha \geq n+1 $ and $ x \in \tilde{\mathcal{S}}_n^+ $, then for $ t \geq 0 $, the determinant of $ X $ satisfies the stochastic differential equation
\begin{align}\label{det Wishart SDE}
	\begin{split}
		\ln\left(\det (\mat X_t)\right) &= \int_{0}^{t} \left( \alpha - n - 1\right) \tr\left( \mat a^\top \mat a  \mat X^{-1}_s \right) + 2 \tr\left( \mat b \right) d s  \\ & \qquad + \int_{0}^{t} 2 \tr\left( \sqrt{\mat X^{-1}_s} d \mat W_s \mat a \right), \qquad 0 \leq t < \tau,
	\end{split}
\end{align}
where $ \tau = \inf \left\{ t \geq 0 : \det{(\mat X_t)} = 0 \right\} $. It was shown in \cite{MAYERHOFER20112072} Theorem 3.4 that for $ \alpha \geq n+1 $ and $ x \in \mathcal{S}_{n}^+ $, $ \tau = \infty $ almost surely. 

Given $ t \geq 0 $, the Laplace transform of $ X_t $ can be computed directly from solving the matrix Riccati ordinary differential equation (see \cite{Ahdida2013Exact} Proposition 4) and is given by
\begin{align*}
	 \E e^{ - \tr\left(u X_t\right)} = \frac{\exp\left\{ \tr \left[ u \left( I - 2 \sigma_t u \right)^{-1} e^{bt} x e^{bt}  \right] \right\}}{\det \left(I - 2 \sigma_t u  \right)^{\alpha / 2}}, \quad u \in \tilde{\mathcal{S}}^+_n,
\end{align*}
where $ \sigma_t = \int_0^t e^{b s} a^\top a e^{b^\top s} ds $ . Therefore, by comparing the above expression to the Laplace transform of the non-central Wishart random variable computed in \cite{Letac20081393}, we deduce that $ X_t $ follows the non-central Wishart distribution with $ \alpha $ degrees of freedom, covariance matrix $ \sigma_t $ and non-centrality matrix $ e^{bt} x e^{bt} \sigma_t^{-1} $, denoted $ \mathcal{W}_n(\alpha, \sigma_t, e^{bt} x e^{bt} \sigma_t^{-1} ) $.

We denote the space of $ n \times n $ matrix-valued continuous function defined on $ [0,t] $ by $ \mathcal{C}\left( [0, \infty), \reals^{n \times n} \right) $, the law of a Wishart process $ X $ on $ \mathcal{C}\left( [0, \infty), \reals^{n \times n} \right) $ and its respective semi-group by $  ^n Q^{\alpha, a, b}_{x} $, or simply $ Q^{\alpha, a, b}_{x} $ when there is no ambiguities about the dimension $ n $. Moreover, we assume $ \Omega = \mathcal{C}\left([0,\infty), \reals^{n \times n}\right) $, the set of $ \reals^{n \times n} $-valued continuous functions defined on $ [0,\infty) $, and denote $ X $ the coordinate process $ X_t(\omega) = \omega_t $.

For $ \alpha \geq n+1 $, the Wishart law $ Q^{\alpha, a, b}_{x} $ is absolutely continuous with respect to the parameters $ \alpha $ and $ b $, their respective Cameron-Martin-Girsanov formulae are given as follows:

\begin{lemma}[Absolute continuity of Wishart laws] \label{CMG formulae}
	Let $ \alpha \geq n+1 $, $ t \geq 0 $ and $ Q^{\alpha, a, b}_x $ be the law of $ WIS(n, \alpha, a, b, x) $ on $ \mathcal{C}([0,\infty), \reals^{n \times n}) $.
	\begin{enumerate}[(i)]
		\item For $ u \in\tilde{\mathcal{S}}_n^- $ such that $ u a = a u $,
		\begin{align}\label{CMG drift}
			\begin{split}
				d Q^{\alpha, a, b + u}_x = &\exp \left\{  \tr \left[ \frac{1}{2}( a^\top a )^{-1} u \left( X_t - X_0 \right) \right. \right. \\ & \left. \left. - \frac{1}{2}\alpha u t - \int_0^t \left(a^\top a \right)^{-1} \left(u^2 + b u\right) X_s ds \right] \right\} d Q^{\alpha, a, b}_x .
			\end{split}
		\end{align}
		\item For $ x \in \mathcal{S}^+_n $ and $ \nu \in [\left(n+1 - \alpha\right)/2, \infty ) $,
		\begin{align}\label{CMG index}
			\begin{split}
				d Q_x^{\alpha + 2 \nu, a, b} = & \left(\frac{\det X_t}{\det x}\right)^{\nu/2} \exp \bigg\{ - \tr \bigg[ \nu b t   \\ &+  \left( \alpha - n - 1 + \nu \right) \frac{\nu}{2} \int_0^t \left(a^\top a\right) X^{-1}_s d s  \bigg] \bigg\} d Q^{\alpha, a, b}_x.
			\end{split}
		\end{align}
	\end{enumerate}
\end{lemma}

\subsection*{Main result}

This article is concerned with the joint conditional Laplace transform of the pair for 
\begin{align} \label{pair}
	\left( \int_0^t X_s ds, \int_0^t \tr \left(a^{-1}a X^{-1}_s \right) ds \right),
\end{align}
for a $ WIS(n,\alpha,a,b,x) $ process $ X $ and $ \alpha \geq n+1 $, given $ X_t $ for a fixed $ t \geq 0 $. 

Let us first state the main result of this article,
\begin{theorem} \label{main result}
Let $ X $ be a $ WIS(n,\alpha,a,b,x) $ process and $ \alpha \geq n+1 $, then
\begin{align} \label{joint conditional Laplace transform}
	\begin{split}
		&\E  \left( \left. \exp \left\{ - \tr \left[ u^2 \int_0^t X_s ds \right] - \frac{\lambda^2}{2} \tr \left[ \int_0^t \left(a^\top a\right) X^{-1}_s d s \right] \right\} \right| X_t = y \right) \\
		= & \frac{q^{\alpha + 2 \nu_\lambda, a, b + \delta_u}_t(x,y)}{q^{\alpha, a, b}_t(x,y)} \left(\frac{\det y}{\det x}\right)^{-\nu_\lambda/2} \\ & \quad  \exp \left\{ \tr \left[ \nu_\lambda b t +  \left(\frac{1}{2}\left(a^\top a\right)^{-3/2}\left(u^2 + bu\right)^{1/2} \left(y - x\right) - \alpha t \right) \right] \right\}, 
	\end{split}
\end{align}
	where
	\begin{align*}
	    u & \in \mathcal{D}, \quad \lambda \in \reals,\\
	    \delta_u &= \frac{1}{2}\left(-b + \sqrt{ b^2 - 4 a^\top a u^2}\right),\\
		\mathcal{D} &= \left\{ u \in \mathcal{S}_n : \delta_u + b \in \tilde{\mathcal{S}}_n^+, au = ua \right\}, \\
		\nu_\lambda &= \sqrt{\lambda^2 + (\alpha - n - 1)^2} - \alpha + n + 1,
	\end{align*}
	and $ q^{\alpha, a, b}_t(x,y) $ denotes the density of a $ WIS(n, \alpha, a, b, x) $ semi-group.
\end{theorem}

Formula (\ref{joint conditional Laplace transform}) is an extension of that given in Proposition 2.4 of \cite{donati2004some}, where $ a $ is assumed to be the identity matrix $ \mathbf{id} $ and $ b $ is $ 0 $. The proof for Theorem \ref{main result} relies on, as in that of \cite{donati2004some}, the absolute continuity of Wishart law with respect to the dimension parameter $ \alpha $ and the drift parameter $ b $ as well as the law of a Wishart bridge process over $ \left[0,t\right] $, which will be defined in the next section.

\section{Wishart bridge processes}

A bridge of a Wishart process can be thought of as a Wishart process with its two end points ``pinned down'' over a fixed time interval. We define the law of a Wishart bridge process as a regular conditional probability measure, analogous to that of a squared Bessel bridge process as defined in \cite{revuz1999continuous} Chapter XI.

We denote the space of $ n \times n $ matrix-valued continuous function defined on $ A \subseteq [0,\infty) $ by $ \mathcal{C}\left( A, \reals^{n \times n} \right) $, the law of $ X $ on $ \mathcal{C}\left( [0, \infty), \reals^{n \times n} \right) $ and its respective semi-group by $  ^n Q^{\alpha, a, b}_{x} $, or simply $ Q^{\alpha, a, b}_{x} $ when there is no ambiguities about the dimension $ n $. Throughout this article, we assume $ \Omega = \mathcal{C}\left([0,\infty), \reals^{n \times n}\right) $ and denote $ X $ the coordinate process $ X_t(\omega) = \omega_t $.

For every $ t \geq 0 $, let us consider the space $ \mathbb{W}_t = \mathcal{C}([0,t], \reals^{n \times n}) $ endowed with the topology generated by the uniform metric $ \rho $ and the Borel $ \sigma $-algebra $ \borel(\mathbb{W}_t) $ generated by this topology. Therefore the metric space $ \left(\mathbb{W}_t, \rho\right) $ is complete and separable (see \cite{billingsley1968convergence}). Consequently, there exists a unique regular conditional distribution of $ ^n Q^{\alpha, a ,b}_x ( {} \cdot | X_t) $, namely a family of probability measures $ ^n Q^{\alpha, a ,b}_{x,y,t} $ on $ \mathbb{W}_t $ such that for every $ B \in \mathcal{B} $,
\begin{align*}
	^n Q^{\alpha, a ,b}_x (B) = \int {^n Q^{\alpha, a ,b}_{x,y,t}}(B) \mu_t(dy),
\end{align*}
where $ \mu_t $ is the density of $ X_t $ under $ ^n Q^{\alpha, a ,b}_x $.

Therefore we can define a Wishart bridge process by specifying its law as follow:

\begin{mydef}
	A continuous process of which law is $ ^n Q^{\alpha, a ,b}_{x,y,t} $ is called an $ n $-dimensional {\it Wishart Bridge process} (with parameters $ \alpha, a, b $) from $ x $ to $ y $ over $ [0,t] $ and is denoted by $ WIS^{n, \alpha, a, b}_t(x,y) $.
\end{mydef}

As for the law of a Wishart process, we simply write $ Q^{\alpha, a ,b}_{x,y,t} $ for the law of a Wishart bridge process when there is no ambiguities about the dimension. Loosely speaking, the law of a Wishart bridge can be understood in a sense that for every $ B \in \borel(\mathbb{W}_t) $,
\begin{align*}
	Q^{\alpha, a ,b}_{x,y,t}(B) = Q^{\alpha, a ,b}_{x}(B | X_t = y),
\end{align*}
where $ X $ is the coordinate process.

From the definition of a regular conditional probability (see, for example \cite{ikeda1989stochastic}), we observe that for every $ B \in \borel(\mathbb{W}_t) $, the map $ y \mapsto Q^{\alpha, a ,b}_{x,y,t}(B) $ is measurable and for every measurable function $ f $ on $ \mathbb{W}_t \times \reals^{n \times n} $,
\begin{align}\label{Wishart bridge property}
	\int f(\omega, \omega_t) Q^{\alpha,a,b}_x\left(d\omega\right) = \int \int f(\omega,y) Q^{\alpha,a,b}_{x,y,t}\left(d \omega\right) \mu_t\left(dy\right).
\end{align}

Throughout this article, we follow the notation in \cite{revuz1999continuous} Chapter III, denoting a semi-group $ P_t $ acting on an element $ f  $ in  $ \mathcal{C}_0\left(\reals^{n \times n}, \reals \right) $  by $ P_{t} f $, that is 
\begin{align*}
	P_{t} f = \int f (y) P_{t}(x,dy),
\end{align*}
where $ \mathcal{C}_0\left(\reals^{n \times n}, \reals \right) $ denotes the set of real-valued continuous functions on $ \reals^{n \times n} $ vanishing at infinity. And the function $ p_{t}(x,y) $ such that 
\begin{align*}
	\int f(y) P_{t}(x,dy) = \int f(y) p_{t}(x,y)dy,
\end{align*}
for every Borel measurable function $ f $ is called the {\it density} of the semi-group $ P_t $.

We also make use of the square bracket $ P_{t} \left[f\right] $ instead of $ P_{t}\left(f\right) $ to avoid confusion with probability measures.

\subsection{Integrated Wishart bridge processes}

Suppose $ X $ is a Wishart process with law $ Q^{\alpha, a, b}_{x} $, we call the process $ Y $ defined by
\begin{align*}
	Y_t = \int_0^t X_s ds, \quad t \geq 0,
\end{align*} 
an {\it integrated Wishart process}. An explicit formula for the conditional Laplace transform of $ Y_t $ given $ X_t $ at a fixed $ t \geq 0 $ was derived in \cite{donati2004some} for $ \alpha \geq n+1 $, $ a = \mathbf{id} $ and $ b = 0 $ using the absolute continuity property of Wishart laws. Similarly, the aforementioned formula can be extended to a more general class of Wishart processes by using the absolute continuity property of Wishart laws.

\begin{theorem} \label{integrated Wishart bridge main result}
	Let $ \alpha \geq n+1 $, $ a \in GL(n) $ and $ b \in \tilde{\mathcal{S}}_n^- $ be commutative. Then for $ t \geq 0 $,
	\begin{align}
		\nonumber &Q_{x, y}^b  \left[ \exp \left\{ - \tr \left[ \left(u^2 + b u \right) \int_0^t \left(a^\top a \right)^{-1} X_s ds \right] \right\} \right] \\ &= \frac{q_t^{b+u}(x,y)}{q_t^b(x,y)} \exp \left\{ \tr \left[ -\frac{1}{2} u \left(\left(a^\top a\right)^{-1}\left(y - x\right) - \alpha t \right) \right] \right\}, \quad u \in \mathcal{D},
	\end{align}
	where
	\begin{align*}
	\mathcal{D} = \left\{ u \in \mathcal{S}_n : u+b \in \tilde{\mathcal{S}}_n^-, au = ua \right\},
	\end{align*}
	and $ X $ is the coordinate process, $ Q^b_{x,y} $ and $ q_t^b $ denotes the $ WIS^{n,\alpha,a,b}_t(x,y) $ law and the density of a $ WIS(n, \alpha, a, b, x) $ semigroup respectively.
\end{theorem}

\begin{proof}
	For every measurable Borel measurable function $ f $, it follows from (\ref{Wishart bridge property}) and the Cameron-Martin-Girsanov formula (\ref{CMG drift}) that
	\begin{align*}
		&\int Q_{x,y}^b \left[ \exp \left\{ - \int_0^t \left(a^\top a \right)^{-1} \left(u^2 + b u\right) X_s ds \right\} \right] f(y) q^b_t (x,y) dy \\
		=& Q_{x}^b \left[\exp \left\{ - \int_0^t \left(a^\top a \right)^{-1} \left(u^2 + b u\right) X_s ds \right\} f\left(X_t\right)\right]\\
		=& Q^{b + u}_x \left[ \exp \left\{ \tr \left[ -\frac{1}{2}u \left(\left(a^\top a\right)^{-1} \left(X_t - x\right) - \alpha t \right) \right] \right\} f(X_t) \right]\\
		=& \int Q^{b + u}_{x,y} \left[ \exp \left\{ \tr \left[ -\frac{1}{2}u \left(\left(a^\top a\right)^{-1}\left(X_t - x\right) - \alpha t \right) \right] \right\} \right] f(y) q^{b+u}_t (x,y) d y \\
		=& \int \exp \left\{ \tr \left[ -\frac{1}{2} u \left(\left(a^\top a\right)^{-1}\left(y - x\right) - \alpha t \right) \right] \right\} f(y) q^{b+u}_t (x,y) d y.
	\end{align*}
	Therefore, we have
	\begin{align*}
	&Q_{x,y}^b \left[ \exp \left\{ - \tr \left(\int_0^t \left(u^2 + bu\right) \left(a^\top a \right)^{-1} X_s\right) ds \right\} \right] q^b_t (x,y) \\
	=& \exp \left\{ \tr \left[ -\frac{1}{2}u \left(\left(a^\top a \right)^{-1} \left(y - x\right) - \alpha t \right) \right] \right\} q^{b+u}_t (x,y),
	\end{align*}
	almost surely.
\end{proof}

Replacing $ u^2 + bu $ in Theorem \ref{integrated Wishart bridge main result} with $ u^2 $ and solve
\begin{align*}
u^2 = \left(a^\top a\right)^{-1}\left( \delta_u^2 + b \delta_u \right),
\end{align*}
for $ \delta_u $, we obtain the followings,

\begin{corollary} \label{Corollary to integrated Wishart bridge main result}
	Let  $ \alpha \geq n+1 $, $ a \in GL(n) $ and $ b \in \tilde{\mathcal{S}}_n^- $ be commutative. Then for $ t \geq 0 $,
	\begin{align}\label{Formula Corollary to Laplace transform integrated Wishart bridge}
		 \nonumber & Q_{x, y}^b \left[ \exp \left\{ - \tr \left(u^2 \int_0^t X_s ds\right) \right\} \right] \\  &=  \frac{q_t^{b+\delta_u}(x,y)}{q_t^b(x,y)}   \exp \left\{ \tr \left[ \frac{1}{2}(a^\top a)^{-3/2}\left(u^2 + bu\right)^{1/2} u \left(y - x - \alpha t \right) \right] \right\}, \quad u \in \mathcal{D},
	\end{align}
	where
	\begin{align*}
	\delta_u &= \frac{1}{2}\left(-b + \sqrt{ b^2 - 4 a^\top a u^2}\right),\\
	\mathcal{D} &= \left\{ u \in \mathcal{S}_n : \delta_u + b \in \tilde{\mathcal{S}}^-_n, au = ua  \right\},
	\end{align*}
	and $ X $ is the coordinate process, $ Q^b_{x,y} $ and $ q_t^b $ denotes the $ WIS^{n,\alpha,a,b}_t(x,y) $ law and the density of a $ WIS(n, \alpha, a, b, x) $ semigroup respectively.
\end{corollary}

In the case of $ WIS(n, \alpha, I_n, 0, x) $, as considered in \cite{donati2004some}, Corollary \ref{Corollary to integrated Wishart bridge main result} allows us to find an explicit expression for the Laplace transform of an integrated Wishart bridge process. This extends formula (2.8) of \cite{donati2004some}, where the Laplace transform  of the trace of an integrated Wishart bridge process was considered. We summarise this result in the corollary below, which can also be considered as the matrix extension of formula (2.m) of \cite{Pitman1982}.

\begin{corollary}
	Let $ \alpha \geq n+1 $. For every $ u \in \tilde{\mathcal{S}}^-_n $, 
	\begin{align*}
		 Q_{x, y} \left[ \exp \left\{ - \tr \left(u^2 \int_0^t X_s ds\right) \right\} \right] =  \frac{q_t^{u}(x,y)}{q_t^0(x,y)} \exp \left( \tr \left[ - \frac{u}{2} \left( \alpha t + x - y \right) \right] \right),
	\end{align*}
where $ Q_{x,y} $ and $ q^b_t(x,y) $ denote the $ WIS^{n, \alpha, I, b}_t(x,y) $ law and the density of a $ WIS(n, \alpha, I, 0, x) $ semigroup respectively.
\end{corollary}

\subsection{Generalised Hartman-Watson law}

The generalised Hartman-Watson law of a Wishart process $ X $ for $ a = \id $ and $ b = 0 $, namely the conditional distribution of 
\begin{align*}
	\tr\left( \int_0^t X^{-1}_s ds \right),
\end{align*}
given $ X_t $, was studied in \cite{donati2004some} through its Laplace transform. By using the Wishart bridge processes and absolute continuity property of Wishart laws, the Laplace transform of the generalised Hartman-Watson law given in \cite{donati2004some} can also be obtained for $ a \neq \id $.

As in Theorem \ref{integrated Wishart bridge main result}, by the definition of Wishart bridge processes and the Cameron-Martin-Girsanov formula (\ref{CMG index}), we have the followings,

\begin{theorem} \label{Hartman-Watson law main result}
	Let $ \alpha \geq n+1 $, $ \nu \in [\left(n+1 - \alpha\right)/2, \infty ) $ and $ t \geq 0 $, then
	\begin{align*}
	 Q^\alpha_{x,y}& \left[ \exp \left\{ - \left( \alpha - n - 1 + \nu \right) \frac{\nu}{2} \tr \left[ \left(a^\top a\right) \int_0^t X^{-1}_s d s  \right] \right\} \right] \\ & = \frac{q_t^{\alpha + 2 \nu}(x,y)}{q_t^\alpha(x,y)} \left(\frac{\det y}{\det x}\right)^{-\nu/2} \exp\{ \nu \tr(b) t \},
	\end{align*}
	where $ Q^\alpha_{x,y} $ and $ q^\alpha_t(x,y) $ denote the $ WIS^{n, \alpha, a, b}_t(x,y) $ law and the density of a $ WIS(n, \alpha, a, b, x) $ semigroup respectively.
\end{theorem}

\begin{proof}
As in the proof of Theorem \ref{integrated Wishart bridge main result} by applying (\ref{CMG index}) and (\ref{Wishart bridge property}).
\end{proof}

We can therefore compute the Laplace transform of the generalised Hartman-Watson law, which extends Proposition 2.4 of \cite{donati2004some} to a wider class of Wishart processes.

\begin{corollary}
	Let $ \alpha \geq n + 1 $ and $ t \geq 0 $, then for every $ u \in \reals $,
	\begin{align*}
	Q^\alpha_{x,y}  \left[ \exp \left\{ - \frac{u^2}{2} \tr \left[ \left(a^\top a\right) \int_0^t X^{-1}_s d s  \right] \right\} \right] = \frac{q_t^{\alpha + 2 \nu_u}(x,y)}{q_t^\alpha(x,y)} \left(\frac{\det y}{\det x}\right)^{-\frac{\nu_u}{2}} \exp\{ \nu_u \tr(b) t \},
	\end{align*}
	where
	\begin{align*}
	\nu_u = \sqrt{u^2 + (\alpha - n - 1)^2} - \alpha + n + 1,
	\end{align*} 
	$ Q^\alpha_{x,y} $ and $ q_t^\alpha $ denote the $ WIS^{n,\alpha,a,b}_t(x,y) $ law and the density of a $ WIS(n, \alpha, a, b, x) $ semigroup respectively.
\end{corollary}

\subsection{Proof of Theorem \ref{main result}}

Combining the arguments made in the proofs of Theorem \ref{integrated Wishart bridge main result} and Theorem \ref{Hartman-Watson law main result}, we have the following,

\begin{theorem} \label{main theorem}
	Let $ a \in GL(n) $, $ b \in \tilde{\mathcal{S}}_n^- $ be commutative, $ \alpha \geq n + 1 $. Then for every $ u \in \tilde{\mathcal{S}}_n $ such that $ u + b \in \tilde{\mathcal{S}}_n^- $, $ u a = a u $ and $ \lambda \in \reals $,
	\begin{align*}
	&Q_{x, y, t}^{\alpha, a, b} \left[ \exp \left\{ - \tr \left[ \left(u^2+bu\right) \int_0^t \left(a^\top a\right)^{-1} X_s ds \right] - \frac{\lambda^2}{2} \tr \left[ \int_0^t \left(a^\top a\right) X^{-1}_s d s \right] \right\} \right] \\
	&= \frac{q^{\alpha + 2 \nu_\lambda, a, b+u}_t(x,y)}{q^{\alpha, a, b}_t(x,y)} \left(\frac{\det y}{\det x}\right)^{-\frac{\nu_\lambda}{2}}  \exp \left\{ \tr \left[ \nu_\lambda b t -\frac{1}{2}  u \left(\left(a^\top a\right)^{-1}\left(y - x\right) - \alpha t \right) \right] \right\},
	\end{align*}
	where
	\begin{align*}
		\nu_\lambda = \sqrt{\lambda^2 + (\alpha - n - 1)^2} - \alpha + n + 1,
	\end{align*}
	$ Q^b_{x,y} $ and $ q_t^b $ denotes the $ WIS^{n,\alpha,a,b}_t(x,y) $ law and the density of a $ WIS(n, \alpha, a, b, x) $ semigroup respectively.
\end{theorem}

Therefore, Theorem \ref{main theorem} can be reformulated to give the joint Laplace transform of the pair 
\begin{align*}
	\left( \int_0^t X_s ds, \int_0^t \tr \left(a^{-1}a X^{-1}_s \right) ds \right),
\end{align*}
under the Wishart bridge law.

\begin{corollary} \label{Corollary to main theorem}
	Let $ a \in GL(n) $, $ b \in \tilde{\mathcal{S}}_n^- $ be commutative, $ \alpha \geq n + 1 $. Then,
	\begin{align*}
	Q_{x, y, t}^{\alpha, a, b} & \left[ \exp \left\{ - \tr \left[ u^2 \int_0^t X_s ds \right] - \frac{\lambda^2}{2} \tr \left[ \int_0^t \left(a^\top a\right) X^{-1}_s d s \right] \right\} \right] \\
	= & \frac{q^{\alpha + 2 \nu_\lambda, a, b + \delta_u}_t(x,y)}{q^{\alpha, a, b}_t(x,y)} \left(\frac{\det y}{\det x}\right)^{-\nu_\lambda/2} \\ & \quad \qquad \exp \left\{ \tr \left[ \nu_\lambda b t +  \left(\frac{1}{2}\left(a^\top a\right)^{-3/2}\left(u^2 + bu\right)^{1/2} \left(y - x\right) - \alpha t \right) \right] \right\}, 
	\end{align*}
	where
	\begin{align*}
	    u & \in \mathcal{D}, \quad \lambda \in \reals,\\
	    \delta_u &= \frac{1}{2}\left(-b + \sqrt{ b^2 - 4 a^\top a u^2}\right),\\
		\mathcal{D} &= \left\{ u \in \mathcal{S}_n : \delta_u + b \in \tilde{\mathcal{S}}_n^-, au = ua \right\}, \\
		\nu_\lambda &= \sqrt{\lambda^2 + (\alpha - n - 1)^2} - \alpha + n + 1,
	\end{align*}
	$ Q^b_{x,y} $ and $ q_t^b $ denote the $ WIS^{n,\alpha,a,b}_t(x,y) $ law and the density of a $ WIS(n, \alpha, a, b, x) $ semi-group respectively.
\end{corollary}

Given a filtered probability space $\filteredspace$ and a Wishart process $X$ defined on it. Then Theorem \ref{main result} follows from Corollary \ref{Corollary to main theorem} by identifying $ \E\left( {}\cdot{} | X_t = y\right) $ to $ \E\left( {}\cdot{} | \sigma (X_t) \right)(\omega_y) $ where $ \omega_y \in \left\{ \omega \in \Omega : X_t(\omega) = y \right\} $.

\newpage

\bibliographystyle{JasonBib}

\bibliography{library}

\end{document}